\documentclass[a4paper,12pt]{article}

\usepackage{amsfonts}
\usepackage{amscd,color}
\usepackage{amsmath,amsfonts,amssymb,amscd}
\usepackage{indentfirst,graphicx,epsfig}
\usepackage{graphicx}
\input{epsf}
\usepackage{epstopdf}
\usepackage{caption}
\usepackage{ifpdf}
\usepackage{fullpage}
\usepackage{epsfig}
\usepackage{latexsym}
\usepackage{caption}
\usepackage{amsmath,amsthm}
\usepackage{amsfonts}
\usepackage{amssymb}
\usepackage{graphicx}
\usepackage{graphics}
\usepackage{bm}
\usepackage{color}
\usepackage{subfigure}
\graphicspath{{\figs}}

\makeatletter

\newcommand{\Rmnum}[1]{\expandafter\@slowromancap\romannumeral #1@}
\makeatother
\makeatletter
\let\@fnsymbol\@arabic
\makeatother

\begin{document}
\newtheorem{theorem}{Theorem}[section]
\newtheorem{observation}[theorem]{Observation}
\newtheorem{corollary}[theorem]{Corollary}
\newtheorem{algorithm}[theorem]{Algorithm}
\newtheorem{problem}[theorem]{Problem}
\newtheorem{question}[theorem]{Question}
\newtheorem{lemma}[theorem]{Lemma}
\newtheorem{proposition}[theorem]{Proposition}
\newtheorem{definition}[theorem]{Definition}
\newtheorem{guess}[theorem]{Conjecture}
\newtheorem{claim}[theorem]{Claim}
\newtheorem{example}[theorem]{Example}
\makeatletter
  \newcommand\figcaption{\def\@captype{figure}\caption}
  \newcommand\tabcaption{\def\@captype{table}\caption}
\makeatother

\newtheorem{acknowledgement}[theorem]{Acknowledgement}

\newtheorem{axiom}[theorem]{Axiom}
\newtheorem{case}[theorem]{Case}
\newtheorem{conclusion}[theorem]{Conclusion}
\newtheorem{condition}[theorem]{Condition}
\newtheorem{conjecture}[theorem]{Conjecture}
\newtheorem{criterion}[theorem]{Criterion}
\newtheorem{exercise}[theorem]{Exercise}
\newtheorem{notation}[theorem]{Notation}
\newtheorem{solution}[theorem]{Solution}
\newtheorem{summary}[theorem]{Summary}
\newtheorem{fact}[theorem]{Fact}
\newtheorem{remark}{Remark}

\newcommand{\pp}{{\it p.}}
\newcommand{\de}{\em}
\newcommand{\mad}{\rm mad}

\newcommand*{\QEDA}{\hfill\ensuremath{\blacksquare}}
\newcommand*{\QEDB}{\hfill\ensuremath{\square}}

\newcommand{\qf}{Q({\cal F},s)}
\newcommand{\qff}{Q({\cal F}',s)}
\newcommand{\qfff}{Q({\cal F}'',s)}
\newcommand{\f}{{\cal F}}
\newcommand{\ff}{{\cal F}'}
\newcommand{\fff}{{\cal F}''}
\newcommand{\fs}{{\cal F},s}
\newcommand{\g}{\gamma}
\newcommand{\wrt}{with respect to }

\title{\bf Upper bounds for the $MD$-numbers and characterization of extremal graphs\footnote{Supported by NSFC No.11871034 and 11531011.}}

\renewcommand{\thefootnote}{\arabic{footnote}}

\author{\small Ping Li, Xueliang Li\\
\small Center for Combinatorics and LPMC\\
\small  Nankai University, Tianjin 300071, China\\
\small Email: qdli\underline{ }ping@163.com, lxl@nankai.edu.cn\\
}

\date{}
\maketitle

\begin{abstract}
For an edge-colored graph $G$, we call an edge-cut $M$ of $G$ monochromatic if the edges of $M$ are colored with the same color.
The graph $G$ is called monochromatic disconnected if any two distinct vertices of $G$ are separated by a monochromatic edge-cut.
For a connected graph $G$, the monochromatic disconnection number
(or $MD$-number for short) of $G$, denoted by $md(G)$, is the maximum number of colors that are allowed in order to make $G$
monochromatic disconnected. For graphs with diameter one, they are complete graphs and so
their $MD$-numbers are $1$. For graphs with diameter at least 3, we can construct $2$-connected graphs such that their
$MD$-numbers can be arbitrarily large; whereas for graphs $G$ with diameter two, we show that if $G$ is a $2$-connected graph then
$md(G)\leq 2$, and if $G$ has a cut-vertex then $md(G)$ is equal to the number of blocks of $G$. So, we will focus on studying
$2$-connected graphs with diameter two, and give two upper bounds of their $MD$-numbers depending on their connectivity and
independent numbers, respectively. We also characterize the $\left\lfloor\frac{n}{2}\right\rfloor$-connected graphs (with large connectivity)
whose $MD$-numbers are $2$ and the $2$-connected graphs (with small connectivity) whose $MD$-numbers
archive the upper bound $\left\lfloor\frac{n}{2}\right\rfloor.$
For graphs with connectivity less than $\frac n 2$, we show that if the connectivity of a graph is in linear with its order $n$, then
its $MD$-number is upper bounded by a constant, and this suggests us to leave a conjecture that for a $k$-connected graph $G$,
$md(G)\leq \left\lfloor\frac{n}{k}\right\rfloor$. \\[2mm]
{\bf Keywords:} monochromatic disconnection number, connectivity, diameter, independent number, upper bound, extremal graph.\\[2mm]
{\bf AMS subject classification (2020)}: 05C15, 05C40, 05C35.
\end{abstract}

\baselineskip16pt

\section{Introduction}

Let $G$ be a graph and let $V(G)$, $E(G)$ denote the vertex-set and the edge-set of $G$, respectively.
We use $|G|$ and $||G||$ to denote the number of vertices and the number of edges of $G$, respectively,
and call them the order and the size of $G$.
If there is no confusion, we also use $n$ and $m$ to denote $|G|$ and $||G||$, respectively, throughout this paper.
Let $S$ and $F$ be a vertex subset and an edge subset of $G$, respectively. Then $G-S$ is the graph obtained from $G$
by deleting the vertices of $S$ together with the edges incident with vertices of $S$, and $G-F$ is the graph whose
vertex-set is $V(G)$ and edge-set is $E(G)-F$. Let $G[S]$ and $G[F]$ be the subgraphs of $G$ induced, respectively,
by $S$ and $F$. We use $[r]$ to denote the set $\{1,2,\cdots,r\}$ of positive integers.
If $r=0$, then set $[r]=\emptyset$. For all other terminology and notation not defined here we follow Bondy and Murty \cite{B}.

For a graph $G$, let $\Gamma: E(G)\rightarrow [r]$ be an {\em edge-coloring} of $G$ that allows a
same color to be assigned to adjacent edges. For an edge $e$ of $G$, we use $\Gamma(e)$ to denote the color of $e$.
If $H$ is a subgraph of $G$, we also use $\Gamma(H)$ to denote the set of colors
on the edges of $H$ and use $|\Gamma(H)|$ to denote the number of colors in $\Gamma(H)$.
For an edge-colored graph $G$ and a vertex $v$ of $G$, the {\em color-degree}
of $v$, denoted by $d^c(v)$, is the number of colors appearing on the edges incident
with $v$.

The three main colored connection colorings: rainbow connection coloring \cite{CJMZ},
proper connection coloring \cite{BFGM} and proper-walk connection coloring \cite{BBY},
monochromatic connection coloring \cite{CY}, have been well-studied in recent years.
As a counterpart concept of the rainbow connection coloring,
rainbow disconnection coloring was introduced in \cite{CDHHZ} by Chartrand et al. in 2018.
Subsequently, the concepts of monochromatic disconnection coloring and
proper disconnection coloring were also introduced in \cite{LL} and \cite{BCJLW, CLLW}.
We refer to \cite{BL} for the philosophy of studying these so-called global graph colorings.
More details on the monochromatic disconnection coloring can be found in \cite{LLi}.
We will further study this coloring in this paper and get some deeper and stronger results.

For an edge-colored graph $G$, we call an edge-cut $M$
a {\em monochromatic edge-cut} if the edges of $M$ are colored with the same color.
If there is a monochromatic $uv$-cut with color $i$, then we say that  color $i$
{\em separates} $u$ and $v$. We use $C_\Gamma(u,v)$ to denote the set of colors in $\Gamma(G)$
that separate $u$ and $v$, and let $c_\Gamma(u,v)=|C_\Gamma(u,v)|$.

An edge-coloring of a graph is called a {\em monochromatic disconnection coloring}
(or {\em $MD$-coloring for short}) if each pair of distinct vertices of the graph has a monochromatic edge-cut separating them,
and the graph is called {\em monochromatic disconnected}.
For a connected graph $G$, the {\em monochromatic disconnection number
(or $MD$-number for short)} of $G$,
denoted by $md(G)$, is defined as the {\em maximum} number of colors that are allowed in order to
make $G$ monochromatic disconnected. An {\em extremal $MD$-coloring} of $G$ is
an $MD$-coloring that uses $md(G)$ colors. If $H$ is a subgraph of $G$ and $\Gamma$ is an edge-coloring of $G$,
we call $\Gamma$ an edge-coloring {\em restricted} on $H$.

The following terminology and notation are needed in the sequel.
Let $G$ and $H$ be two graphs. The {\em union} of $G$ and $H$ is the graph $G\cup H$ with vertex-set
$V(G)\cup V(H)$ and edge-set $E(G)\cup E(H)$. The {\em intersect} of $G$ and $H$ is the graph $G\cap H$
with vertex-set $V(G)\cap V(H)$ and edge-set $E(G)\cap E(H)$.
The {\em Cartesian product} of $G$ and $H$ is the graph $G\Box H$ with $V(G\Box H)=\{(u,v):u\in V(G),v\in V(H)\}$,
$(u,v)$ and $(x,y)$ are adjacent in $G\Box H$ if either $ux$ is an edge of $G$ and $v=y$, or $vy$ is an edge of $H$ and $u=x$.
If $G$ and $H$ are vertex-disjoint, then let $G\vee H$ denote the {\em join} of $G$ and $H$
which is obtained from $G$ and $H$ by adding an edge between every vertex of $G$ and every vertex of $H$.

For a graph $G$, a {\em pendent vertex} of $G$ is a vertex with degree one.
The {\em ends} of $G$ is the set of pendent vertices,
and the {\em internal vertex set} of $G$ is the set of vertices with degree at least two.
We use $end(G)$ and $I(G)$ to denote the ends of $G$
and the internal vertex set of $G$, respectively.
The {\em independent number} of $G$, denoted by $\alpha(G)$, is the order of a maximum independent set of $G$.
For two vertices $u,v$ of $G$, we use $N(u)$ to denote the {\em neighborhood} of $u$ in $G$,
and $N(u,v)$ to denote the set of common neighbors of $u$ and $v$ in $G$. The distance between $u$ and $v$ in $G$
is denoted by $d(u,v)$, and the diameter of $G$ is denoted by $diam(G)$.
We call a cycle $C$ (path $P$) a {\em $t$-cycle} ({\em $t$-path}) if $|C|=t$ ($||P||=t$).
If $t$ is even (odd), then we call the path an {\em even (odd) path} and the cycle an {\em even (odd) cycle}.
A $3$-cycle is also called a {\em triangle}.
A {\em matching-cut} of $G$ is an edge-cut of $G$, which also forms a matching in $G$.

In \cite{LL, LLi} we got the following results, which are restated for our later use.

\begin{lemma}\label{S-mono-1}\cite{LL}
\begin{enumerate}
\item If a connected graph $G$ has $r$ blocks $B_1,\cdots,B_r$,
      then $md(G)=\sum_{i\in [r]}md(B_i)$ and $md(G)=n-1$ if and only if $G$ is a tree.
\item $md(G)=\lfloor\frac{|G|}{2}\rfloor$ if $G$ is a cycle,
      and $md(G)=1$ if $G$ is a complete graph with order at least two.
\item If $H$ is a connected spanning subgraph of $G$, then $md(H)\geq md(G)$. Thus, $md(G)\leq n-1$.
\item If $G$ is connected, then $md(v\vee G)=1$.
\item If $v$ is neither a cut-vertex nor a pendent
  vertex of $G$ and $\Gamma$ is an extremal $MD$-coloring of $G$, then $\Gamma(G)\subseteq \Gamma(G-v)$, and
  thus, $md(G)\leq md(G-v)$.
\end{enumerate}
\end{lemma}

\begin{theorem} \label{S-2con}\cite{LL}
If $G$ is a $2$-connected graph, then $md(G)\leq \left\lfloor\frac{n}{2}\right\rfloor$.
\end{theorem}

\begin{theorem} \cite{LLi}\label{S-mono-Box}
If $G_1$ and $G_2$ are connected graphs, then $md(G_1\Box G_2)=md(G_1)+md(G_2)$.
\end{theorem}

\begin{lemma}\cite{LLi}\label{S-mono-matching}
If $G$ has a matching-cut, then $md(G)\geq 2$.
\end{lemma}

We will list some easy observations in the following, which will be used
many times throughout this paper. Suppose $\Gamma$ is an $MD$-coloring of $G$.
If $H$ is a subgraph of $G$, then $\Gamma$ is an $MD$-coloring restricted on $H$.
Every triangle of $G$ is monochromatic. If $G$ is a $4$-cycle, then its  opposite
edges have the same color. If $G$ is a $5$-cycle, then there are two adjacent edges
having the same color.

Let $V$ be a set of vertices and let $\mathcal{E}\subseteq 2^V$.
Then a hypergraph $\mathcal{H}=(V,\mathcal{E})$ is a {\em linear hypergraph}
if $|E_i|\geq 2$ and $|E_i\cap E_j|\leq 1$ for any $E_i,E_j\in \mathcal{E}$.
The {\em size} of $\mathcal{H}$ is the number of hyperedges in $\mathcal{H}$.
A {\em hyperedge-coloring} of $\mathcal{H}$ assigns each hyperedge a positive integer.
A linear hypergraph $\mathcal{H}$ (say the size of $\mathcal{H}$ is $k$) is a {\em linear hypercycle}
if there is a sequence of hyperedges of $\mathcal{H}$,
say $E_1,\cdots, E_k$, and there exist $k$ distinct
vertices $v_1,\cdots,v_k$ of $\mathcal{H}$,
such that $E_1\cap E_k=\{v_k\}$ and $E_i\cap E_{i+1}=\{v_i\}$
for $i\in[k-1]$. If we delete a hyperedge from a linear hypercycle
and then delete the vertices only in this hyperedge,
then we call the resulting hypergraph a {\em linear hyperpath}.
A linear hypercycle (linear hyperpath) is called a {\em linear hyper $k$-cycle (linear hyper $k$-path)} if the size of this
linear hypercycle (linear hyperpath) is $k$.

\section{Preliminaries}

We need some more preparations before proceeding to our main results.

\begin{lemma}\label{2-con-md1}
For two connected graphs $G_1$ and $G_2$, if $md(G_1\cap G_2)=1$ then $md(G_1\cup G_2)=md(G_1)+md(G_2)-1$.
\end{lemma}
\begin{proof}
Let $G=G_1\cup G_2$ and $\Gamma$ be an extremal $MD$-coloring of $G$.
Then $|\Gamma(G_1\cap G_2)|=1$ and
$\Gamma$ is an $MD$-coloring restricted on $G_1$ (and also $G_2$).
So, $md(G_1\cup G_2)=|\Gamma(G_1)|+|\Gamma(G_2)|-|\Gamma(G_1\cap G_2)|\leq md(G_1)+md(G_2)-1$.
On the other hand,
since $E(G_1\cap G_2)$ is monochromatic under any $MD$-coloring of $G_1\cup G_2$,
let $\Gamma_i$ be an $MD$-coloring of $G_i$ for $i\in[2]$
such that $\Gamma_1(G_1\cap G_2)=\Gamma_2(G_1\cap G_2)=\Gamma(G_1)\cap \Gamma(G_2)$.
Let $\Gamma'$ be an edge-coloring of $G_1\cup G_2$ such that $\Gamma'(e)=\Gamma_i(e)$
if $e\in E(G_i)$, and let $w$ be a vertex of $G_1\cap G_2$.
Then for any two vertices $u,v$ of $G_1\cup G_2$,
if $u,v\in V(G_i)$, then $C_{\Gamma_i}(u,v)\subseteq C_{\Gamma'}(u,v)$;
if $u\in V(G_1)-V(G_2)$ and $v\in V(G_2)-V(G_1)$,
then $(C_{\Gamma_1}(u,w)\cup C_{\Gamma_2}(v,w))\subseteq C_{\Gamma'}(u,v)$.
So, $\Gamma'$ is an $MD$-coloring of $G$, i.e.,
$md(G_1\cup G_2)\geq |\Gamma(G_1\cup G_2)|=md(G_1)+md(G_2)-1$.
Therefore, $md(G_1\cup G_2)=md(G_1)+md(G_2)-1$.
\end{proof}

\begin{lemma}\label{S-mono-path}
Let $G$ be a connected graph
and let $G'$ be a graph obtained from $G$
by replacing an edge $e=ab$ with a path $P$.
Then $md(G')\geq md(G)+\left\lfloor\frac{||P||-1}{2}\right\rfloor$.
\end{lemma}
\begin{proof}
Let $\Gamma$ be an extremal $MD$-coloring of $G$.
Let $||P||=t$ and
let $P=ae_1c_1\cdots e_tb$.
Let $\Gamma'$ be an edge-coloring of $G'$ such that $\Gamma(f)=\Gamma'(f)$ when $f\in E(G)-e$,
$\Gamma'(e_i)=\Gamma'(e_{t+1-i})=|\Gamma(G)|+i$ for $i\in[\left\lfloor\frac{t-1}{2}\right\rfloor]$,
$\Gamma(e)=\Gamma'(e_{\frac{t+1}{2}})$
when $t$ is odd, and $\Gamma(e)=\Gamma'(e_{\frac{t}{2}})
=\Gamma'(e_{\frac{t}{2}+1})$ when $t$ is even.
It is easy to verify that $\Gamma'$ is an $MD$-coloring of $G'$.
Thus, $md(G')\geq md(G)+\left\lfloor\frac{||P||-1}{2}\right\rfloor$.
\end{proof}

\begin{lemma}\label{super-conn}
Suppose $u,v$ are nonadjacent vertices of $G$
and $\Gamma$ is an extremal $MD$-coloring of $G$.
Let $C_\Gamma(u,v)=\{t\}$ and $e$ an extra edge, and let $\Gamma'$ be an edge-coloring of $G\cup e$ that is
obtained from $\Gamma$ by coloring the added edge $e$ with color $t$.
Then $\Gamma'$ is an $MD$-coloring of $G\cup e$ and $md(G)=md(G\cup e)$.
\end{lemma}
\begin{proof}
Let $H_i$ be the graph obtained from $G$ by deleting all the edges with color $i$.
Let $G'=G\cup e$.
If $\Gamma'$ is not an $MD$-coloring of $G'$,
then there are two vertices $x,y$ of $G'$
such that $C_{\Gamma'}(x,y)=\emptyset$.
If $t\in C_\Gamma(x,y)$, since $x,y$ are in different components of $H_t$,
we have $t\in C_{\Gamma'}(x,y)$, a contradiction.
If $t\notin C_\Gamma(x,y)$,
then let $j\in C_\Gamma(x,y)$.
Then there are two components $D_1,D_2$ of $H_j$
such that $x\in V(D_1)$ and  $y\in V(D_2)$.
Since $j$ does not separate $x,y$ in $G'$,
the edge $e$ connects $D_1$ and $D_2$, say $u\in V(D_1)$ and $v\in V(D_2)$.
Thus, the color $j$ separates $u,v$ in $G$,
which contradicts that $C_\Gamma(u,v)=\{t\}$.
Therefore, $\Gamma'$ is an $MD$-coloring of $G'$.
Since $|\Gamma'(G')|=|\Gamma(G)|$ and
$\Gamma$ is an extremal $MD$-coloring of $G$, we have $md(G')\geq md(G)$.
Since $G$ is a connected spanning subgraph of $G'$,
by Lemma \ref{S-mono-1} (3) we have $md(G)\geq md(G')$.
So, $md(G)=md(G')$.
\end{proof}

Suppose $\Gamma$ is an $MD$-coloring of $G$ and $G_i$
is the subgraph of $G$ induced by the set of edges with color $i$,
which, in what follows, is called the {\em color $i$ induced subgraph} of $G$.
Then for any component $D_1$ of $G_i$ and any component $D_2$ of $G_j$,
we have $|V(D_1)\cap V(D_2)|\leq1$; otherwise, suppose $u,v\in V(D_1)\cap V(D_2)$.
Then $C_\Gamma(u,v)=\emptyset$, a contradiction. We use $\mathcal{H}_\Gamma$ to denote a
hyperedge-colored hypergraph with vertex-set $V(G)$ and hyperedge-set $\{V(D) \ | \ D\mbox{ is a component of some }G_i\}$,
and the hyperedge $F$ has color $i$ if $F$ corresponds to a component of $G_i$.
Let $H_\Gamma$ be a graph with $V(H_\Gamma)=V(G)$ and
$$E(H_\Gamma)=\{uv \ | \ u,v\mbox{ are in the same component of some }G_i\}.$$
Then each hyperedge of $\mathcal{H}_\Gamma$ corresponds to a clique of $H_\Gamma$,
and any two hyperedges of $\mathcal{H}_\Gamma$ (any two cliques of $H_\Gamma$)
share at most one vertex. Thus, $\mathcal{H}_\Gamma$ is a linear hypergraph.
If $F$ is a hyperedge of $\mathcal{H}_\Gamma$ and $u,v\in F$, then
$c_\Gamma(u,v)=1$. According to Lemma \ref{super-conn}, we have the following result.
\begin{lemma}
If $\Gamma$ is an extremal $MD$-coloring of $G$, then $md(G)=md(H_\Gamma)$.
\end{lemma}

Suppose $\Gamma$ is an $MD$-coloring of $G$ and $\mathcal{C}$ is a
hyper $k$-cycle of $\mathcal{H}_\Gamma$. Then there is a $k$-cycle $C$
of $H_\Gamma$ such that any adjacent edges of $C$ have different colors.
Thus, $t\neq 3,5$. Moreover, if $k=4$, then the opposite hyperedges of $\mathcal{C}$ have the same color.

\section{Graphs with diameter two}

In this section, we show that $md(G)\leq 2$ for a $2$-connected graph $G$ if $diam(G)\leq 2$.
However, for any integer $d\geq 3$, we can construct a $2$-connected graph $G$ such that $diam(G)=d$ and $md(G)$ can
be arbitrarily large. Thus, it makes sense to focus on studying the graphs with diameter two, since graphs with diameter 1 are complete
graphs and their $MD$-numbers are 1.

\begin{theorem}\label{S-mono-diam2}
Suppose $G$ is a graph with $diam(G)=2$. Then
\begin{enumerate}
\item if $G$ has a cut-vertex, then
$md(G)$ is equal to the number of blocks of $G$;
\item if $G$ is a $2$-connected graph, then $md(G)\leq 2$;
\item if any two nonadjacent vertices of $G$
  has at least two common neighbors, then $md(G)\leq 2$, and the
  equality holds if and only if $G=K_s\Box K_t$, where $s,t\geq 2$.
\end{enumerate}
\end{theorem}
\begin{proof}
The proof of statement (1) goes as follows. If $v$ is a cut-vertex of $G$ and $diam(G)=2$,
then $v$ connects every vertex of $V(G-v)$. Thus, for each block $D$ of $G$,
$D-v$ is connected and $D=(D-v)\vee v$, i.e., $md(D)=1$.
Therefore, $md(G)$ is equal to the number of blocks of $G$.

Next, for the proof of statement (2) suppose $\Gamma$ is an $MD$-coloring of $G$ with
$|\Gamma(G)|\geq 3$. Then each hypercycle (hyperpath) of the above mentioned hypergraph
$\mathcal{H}_\Gamma$ is a linear hypercycle (linear hyperedge).
We now prove that there is a rainbow hyper $3$-path (the colors of the three hyperedges are pairwise differently)
in $\mathcal{H}_\Gamma$. Since $\mathcal{H}_\Gamma$ does not have hyper $3$-cycle,
the union of three consecutive hyperedges forms a hyper $3$-path.
If every vertex $z$ of $G$ has $d^c(z)\leq 2$, then there is a rainbow hyper $3$-path in $\mathcal{H}_\Gamma$.
If there is a vertex $x$ of $G$ with $d^c(x)\geq 3$,
then there are three hyperedges, say $D_1,D_2$ and $D_3$, such that $x$ is the common vertex of them.
Then the colors of $D_1,D_2$ and $D_3$ are pairwise differently.
Since $G$ is a $2$-connected graph, there is a vertex $w$ of $V(D_1)-\{x\}$ with $d^c(w)\geq 2$
(otherwise, $x$ is a cut-vertex of $G$, a contradiction).
Then there is a hyperedge $F$ of $G$, such that $w$ is a common vertex of $F$ and $D_1$.
Thus, either $F\cup D_1\cup D_2$ or $F\cup D_1\cup D_3$ is a rainbow hyper $3$-path.

Let $\mathcal{P}$ be a rainbow hyper $3$-path of $\mathcal{H}$ and let
$V(D_i)\cap V(D_{i+1})=\{u_i\}$ for $i\in[2]$.
Let $u\in V(D_1)-\{u_1\}$ and $v\in V(D_3)-\{u_2\}$.
We use $\mathcal{P}_{u,v}$ to denote a minimum hyperpath connecting $u$ and $v$.
Since $diam(G)=2$, the size of $\mathcal{P}_{u,v}$ is either one or two.
Let $\mathcal{C}=\mathcal{P}_{u,v}\cup \mathcal{P}$.
If $\mathcal{P}_{u,v}$ is a hyperedge, then $\mathcal{C}$ is a hyper $4$-cycle.
Since $D_1$ and $D_3$ are opposite hyperedges of $\mathcal{C}$ and they have different colors, a contradiction.
If $\mathcal{P}_{u,v}$ is a hyper $2$-path, then let $F_1,F_2$ be hyperedges of $\mathcal{P}_{u,v}$, and let $V(F_1)\cap V(F_2)=\{u_3\}$.
If $u_3\notin \{u_1,u_2\}$, then $\mathcal{C}$ is a hyper $5$-cycle, a contradiction.
If $u_3\in \{u_1,u_2\}$, then $\mathcal{C}$ contains a hyper $3$-cycle, a contradiction.

Finally, we show statement (3). It is obvious that $diam(G)\leq2$, and $G$ is a $2$-connected graph when $n\geq 3$.
So, $md(G)\leq 2$. Suppose $G=K_s\Box K_t$ and $s,t\geq 2$. Then $|N(u,v)|=2$ for any nonadjacent vertices $u$ and $v$ of $G$.
By Lemma \ref{S-mono-1} (2) and Theorem \ref{S-mono-Box}, we have $md(G)=md(K_s)+md(K_t)=2$.

Suppose $md(G)=2$. Then $n\geq 3$ and $G$ is a $2$-connected graph. Let $\Gamma$ be an extremal $MD$-coloring of $G$
and let $G_1,G_2$ be the colors $1, 2$ induced subgraphs of $G$, respectively.
Since $md(G)=2$, we have $d^c(v)\leq 2$ for each $v\in V(G)$.
If $d^c(v)=1$, by symmetry, suppose $v$ is in a component $D$ of $G_1$.
Since $md(G)=2$, we have $D\neq G$, i.e., there exists a vertex $u$ in $V(G)-V(D)$.
Then $u,v$ are nonadjacent and $N(u,v)\subseteq D$. Let $\{a,b\}\subseteq N(u,v)$.
Since $\Gamma(va)=\Gamma(vb)=1$, we have $va\cup vb\cup ua\cup ub$
is a monochromatic $4$-cycle, i.e., $u\in V(D)$, a contradiction.
Thus, $d^c(v)=2$ for each $v\in V(G)$. We use $D_u^1$ and $D_u^2$ to denote the components of $G_1$ and $G_2$,
respectively, such that $V(D_u^1)\cup V(D_u^2)=u$.

Suppose there are $t$ components of $G_1$ and $s$ components of $G_2$.
Since $G$ is a $2$-connected graph, we have $s,t\geq 2$.
Otherwise, if $s=1$, then for each vertex $v$ of $G_1$, $v$ is a cut-vertex, a contradiction.
We label the $t$ components of $G_1$ by the numbers in $[t]$ and label the $s$ components of $G_2$ by the numbers in $[s]$, respectively.
We use $l_1(D)$ to denote the label of a component $D$ of $G_1$,
and use $l_2(F)$ to denote the label of a component $F$ of $G_2$.
For a vertex $u$ of $G$, since $d^c(u)=2$, we use $(l_1(D_u^1),l_2(D_u^2))$ to denote $u$.
For two vertices $u,v$ of $G$, let $u=(i,j)$ and let $v=(s,t)$.
In order to show $G=K_s\Box K_t$, we need to show that $uv$ is an edge of $G$ when $i=s$ and $j\neq t$,
or $i\neq s$ and $j=t$, and $u,v$ are nonadjacent vertices when $i\neq s$ and $j\neq t$.
If $i\neq s$ and $j\neq t$, then $v\notin V(D_u^1\cup D_u^2)$.
Since $N(u)\subseteq V(D_u^1\cup D_u^2)$, $u,v$ are nonadjacent vertices of $G$.
If, by symmetry, $i=s$ and $j\neq t$, then $D_u^1=D_v^1$.
Let $u'\in V(D_u^2)-\{u\}$. Then $u',v$ are nonadjacent.
Since $N(v)\subseteq V(D_v^1\cup D_v^2)$ and $N(u')\subseteq V(D_{u'}^1\cup D_{u'}^2)$, we have
$$2\leq |N(v,u')|\leq|V(D_v^1\cup D_v^2)\cap V(D_{u'}^1\cup D_{u'}^2)|=
|D_v^1\cap D_{u'}^2|+|D_{u'}^1\cap D_v^2|\leq 2.$$
Thus, $D_v^1\cap D_{u'}^2\subseteq N(v,u')$.
Since $D_v^1\cap D_{u'}^2=\{u\}$, we have $uv$ is an edge of $G$.
\end{proof}

\begin{remark} Suppose $L_1,\cdots, L_r$ are $r \ (\geq 2)$ internal disjoint odd paths with
an order $2k_i+2$ for each $i\in [r]$, and they have the same ends $\{u,v\}$.
Let $L_i=ue_1^ix_1^ie_2^ix_2^i\cdots x_{2k_i}^i$ $e_{2k_i+1}v$.
Let $c_0=1$ and $c_i=\Sigma_{j=0}^ik_j$. If $k_i\geq 1$ for $i\in[r]$,
then let $\Gamma$ be an edge-coloring of $G$ such that
$\Gamma(e_j^i)=\Gamma(e_{2k_i+2-j}^i)=c_{i-1}+j$
and $\Gamma(e^i_{k_i+1})=1$ for each $i\in[r]$ and $j\in [k_i]$.
Then $\Gamma$ is an $MD$-coloring of $G$ with $|\Gamma(G)|=\frac{|G|}{2}$.
Since $G$ is a $2$-connected graph, we have $md(G)=\frac{|G|}{2}$.
If $k_i=1$ for each $i\in[r]$, then $G$ is a $2$-connected graph with $diam(G)=3$
and $md(G)=r$. Therefore, there exist $2$-connected graphs with diameter three,
but their $MD$-numbers can be arbitrarily large.
\end{remark}

Let $A_n$ be a graph with $V(A_n)=\{v_1,\cdots,v_{\left\lceil\frac{n}{2}\right\rceil}\}\cup \{u_1,\cdots,u_{\left\lfloor\frac{n}{2}\right\rfloor}\}$
and
$E(A_n)=\{v_iv_j:i,j\in[\left\lceil\frac{n}{2}\right\rceil]\}\cup
\{u_iu_j:i,j\in [\left\lfloor\frac{n}{2}\right\rfloor]\}\cup
\{v_iu_i:i\in [\left\lfloor\frac{n}{2}\right\rfloor]\}$.
Then $\{v_iu_i:i\in [\left\lfloor\frac{n}{2}\right\rfloor]\}$
is a matching-cut of $G$. If $n$ is an odd integer, then let
$$\mathcal{A}_n=\{A_n-E \ | \ E\mbox{ is either an emptyset or a matching of }
 A_n[\{v_1, \cdots,v_{\frac{n-1}{2}}\}]\}.$$

\begin{theorem}\label{S-mono-n2}
Suppose $G$ is a $\left\lfloor\frac{n}{2}\right\rfloor$-connected graph
and $n\geq 4$. Then $md(G)\leq 2$ and
\begin{enumerate}
\item if $n$ is even, then $md(G)=2$ if and only if $G= A_n$;
\item if $n$ is odd, then $md(G)=2$ if and only if $G\in \mathcal{A}_n$.
\end{enumerate}
\end{theorem}

\begin{proof} Since $N(x)+N(y)\geq n-1$ for any two nonadjacent vertices $x$ and $y$,
we have $diam(G)\leq 2$. So, $md(G)\leq 2$.

It is obvious that $G$ is a $\left\lfloor\frac{n}{2}\right\rfloor$-connected
graph if $G=A_n$ or $G\in \mathcal{A}_n$. Moreover, by Lemma \ref{S-mono-matching} and
Theorem \ref{S-mono-diam2}, we have $md(G)=2$.

Now suppose $G$ is a $\left\lfloor\frac{n}{2}\right\rfloor$-connected graph
and $md(G)=2$. Since $n\geq 4$, $G$ is a $2$-connected graph. We distinguish
the following cases for our proof.

{\bf Case 1.} $n$ is even.

For any two nonadjacent vertices $u,v$ of $G$, $|N(u)\cap N(v)|\geq 2$.
By Theorem \ref{S-mono-diam2} (3), $G=K_s\Box K_t$, where $s,t\geq 2$.
We need to prove that at least one of $s,t$ equals two.
Suppose $H_1,H_2$ are two cliques of order $s,t$, respectively, and
$V(H_1)\cap V(H_2)=\{u\}$. Then $N(u)\subseteq V(H_1\cup H_2)$, i.e., $s+t-2\geq \frac{n}{2}$.
Since $n=st$, we have $t(s-2)\leq 2(s-2)$.
Thus, either $s=2$ or $t=2$.

{\bf Case 2.} $n$ is odd.

Say $n=2k+1$ for some integer $k$. Suppose $\Gamma$ is an extremal $MD$-coloring of $G$
and $G_1,G_2$ are the colors $1,2$ induced subgraphs, respectively.

{\em Subcase 2.1} Every vertex $v$ of $G$ has $d^c(v)=2$.

Suppose there are components $D, F$ of $G_1,G_2$,
respectively, such that $V(G)\cap V(F)=\emptyset$.
Then let $u\in V(D)$ and $v\in V(F)$.
Since $d^c(u)=d^c(v)=2$, there are components $D'$ of $G_1$ and $F'$ of $G_2$,
such that $V(D)\cap V(F')=\{u\}$ and $V(F)\cap V(D')=\{v\}$.
Since $V(D)\cup V(F')-\{u\}$ and $V(D')\cup V(F)-\{v\}$ are vertex-cuts of $G$,
we have $|V(D)\cup V(F')|\geq k+1$ and $|V(D')\cup V(F)|\geq k+1$.
Since $|V(D')\cap V(F')|\leq 1$, we have
$n\geq |V(D)\cup V(F')|+|V(D')\cup V(F)|-|V(D')\cap V(F')|\geq 2k+1=n$,
i.e., $D\cup D'\cup F\cup F'=G$.
Then $u$ is a cut-vertex of $G$, a contradiction.
Therefore, for each component $D$ of $G_1$ and each component $F$ of $G_2$,
we have $|V(G)\cap V(F)|=1$.
Then since $d^c(v)=2$ for each $v\in V(G)$, any two components of $G_1$ (and also $G_2$)
have the same order, say $s$ (the order is $t$).
Then $s,t>2$; otherwise, suppose $s=2$, i.e., $G_1$ is a matching.
Since $n$ is odd, we have $V(G)-V(G_1)\neq \emptyset$.
Thus, each vertex $v$ of $V(G)-V(G_1)$ has $d^c(v)=1$, a contradiction.
For a vertex $x$ of $G$, let $D_1,D_2$ be the components of $G_1, G_2$, respectively,
containing $x$. Then $D_1\cup D_2-\{x\}$ is a vertex-cut of $G$, i.e., $s+t-2\geq k$.
However, $2k+1=n=st$ and $s,t>3$, a contradiction.

{\em Subcase 2.2} There is a vertex $v$ of $G$ with $d^c(v)=1$.

Suppose $D$ is the component of $G_1$ containing $v$.
Then since $D-\{v\}$ is a vertex cut of $G$, we have $|D|\geq k+1$.
Since the set of vertices of $D$ with color-degree two is a vertex-cut of $G$,
there are at least $k$ vertices of $D$,
say $v_1,\cdots,v_k$, such that $d^c(v_i)=2$ for $i\in[k]$.
Let $F_i$ be the component of $G_2$ containing $v_i$
and let $U=\bigcup_{i\in[k]}(V(F_i)-\{v_i\})$.
Then $|U|\geq k$. Since $n\geq |D|+|U|\geq 2k+1=n$,
we have $|D|=k+1$, $|U|=k$, and $|F_i|=2$ for $i\in[k]$.
Moreover, $N(v)=\{v_1,\cdots,v_k\}$.
Let $V(F_i)-\{v_i\}=\{u_i\}$. For $i,j\in[k]$, if $u_iu_j$ is not an edge of $G$,
then $U-\{u_i,u_j\}+v_j$ is a vertex-cut of $G$ with order $k-1$,
which contradicts that $G$ is $k$-connected.
For each $v_i$, if there are two vertices $v_j,v_l$
such that $v_iv_j$ and $v_iv_l$ are not edges of $G$, then
$V(D)-\{v_i,v_j,v_l\}+u_i$ is a vertex-cut of $G$ with order $k-1$,
which contradicts that $G$ is $k$-connected.
Therefore, $v_i$ connects all but at most one vertex of $D-v$.
So, $G\in \mathcal{A}_n$.
\end{proof}

\section{Upper bounds}

In this section, we give two upper bounds of the monochromatic disconnection number of a graph $G$, one of which depends on the connectivity
of $G$, and the other depends on the independent number of $G$. Note that for a $k$-connected graph $G$, when $k=2$ (small) and
$k\geq \left\lfloor\frac{n}{2}\right\rfloor$ (large), from Theorems \ref{S-2con} and \ref{S-mono-n2} we know that
$md(G)\leq \left\lfloor\frac{n}{k}\right\rfloor$. This suggests us to make the following conjecture.
\begin{conjecture}
Suppose $G$ is a $k$-connected graph.
Then $md(G)\leq \left\lfloor\frac{n}{k}\right\rfloor$.
\end{conjecture}
Suppose $P$ is a $k$-path. Then $md(K_r\Box P)=md(K_r)+md(P)=k+1$.
Since $n=|K_r\Box P|=r(k+1)$ and $K_r\Box P$ is an $r$-connected graph,
the bound is sharp for $k\geq 2$ if the conjecture is true.

The {\em mean distance} of a connected graph $G$ is defined as $\mu(G)={n\choose 2}^{-1}\Sigma_{u,v\in V(G)}d(u,v)$.
Plesn\'{l}k in \cite{P} posed the problem of finding sharp upper bounds on $\mu(G)$ for $k$-connected graphs.
Favaron et al. in \cite{FKM} proved that if $G$ is a $k$-connected graph of order $n$, then
\begin{align}
\mu(G)\leq \left\lfloor\frac{n+k-1}{k}\right\rfloor\cdot
\frac{n-1-\frac{k}{2}\left\lfloor\frac{n-1}{k}\right\rfloor}{n-1},
\end{align}
and the bound is sharp when $n$ is even.
If $n$ is odd and $k\geq 3$, then Dankelmann et al. in \cite{PSC} proved that
$\mu(G)\leq \frac{n}{2k+1}+30$ and
this bound is, apart from an additive constant, best possible.

The following result gives a relationship between the monochromatic disconnection
number and the connectivity of a graph, which means that if the connectivity of a graph
is in linear of the order of the graph, then the monochromatic disconnection number of the graph
is upper bounded by a constant.

\begin{theorem}
For any $0<\varepsilon<\frac{1}{2}$, there is a constant
$C=C(\varepsilon)<\frac{(1+\varepsilon)^2}{4\varepsilon^2(1-\varepsilon)}$,
such that for any $\varepsilon n$-connected graph $G$, $md(G)\leq C$.
\end{theorem}
\begin{proof}
Suppose $\Gamma$ is an extremal $MD$-coloring of $G$ and $V(G)=\{v_1,\cdots,v_n\}$.
We use $(i,j)$ to denote an unordered integer pair in this proof.
For each color $i$ of $\Gamma(G)$, let
$$S_i=\{(j,l):\mbox{ the color }i\mbox{ sepatates }v_j\mbox{ and }v_l\}.$$
Then $\Sigma_{i\in \Gamma}|S_i|=\Sigma_{j\neq l}c_\Gamma(v_j,v_l)$.

\begin{claim}\label{S-momo-clm1}
$|S_i|\geq k(n-k)$ for each $i\in \Gamma(G)$.
\end{claim}
\begin{proof}
Let $\varepsilon n=k$. The result holds obviously for $k=1$.
Thus, let $k\geq 2$. For each $i\in \Gamma(G)$, let $G_i$ be the color $i$ induced subgraph of $G$,
and let $H_i$ be the graph obtained from $G$ by deleting all the edges with color $i$.
Then $H_i$ is a disconnected graph.
Suppose there is a component $D$ of $H_i$ with $|D|>n-k$.
Let $U=\{v_j \ | \ v_j\in V(D)\cap V(G_i)\}$.
For a component $B$ of $G_i$, if $V(B)\cap V(D)\neq \emptyset$,
then $|V(B)\cap V(D)|=1$.
Since $B$ contains at least one vertex of $V(G-D)$, we have $|U|\leq |V(G-D)|<k$.
Since $|D|>n-k=n(1-\varepsilon)>\varepsilon n=k$,
$U$ is a proper subset of $V(D)$.
So, $U$ is a vertex-cut of $G$.
Since $|U|<k$ and $G$ is $k$-connected, this yields a contradiction.
Thus, for each $i\in \Gamma(G)$,
there is no component of $H_i$ with order greater than $n-k$.

We partition the components of $H_i$ into $r$ parts such that $r$ is minimum and
the number of vertices in each part is at most $n-k$.
Suppose the $r$ parts have $n_1,\cdots,n_r$ vertices, respectively.
Then $\sum_{j\in[r]}n_j=n$.
If $r\geq 4$, then since $r$ is minimum, $n_l+n_j>n-k$ for each $l,j\in [r]$.
Thus,
$$n(r-1)=(r-1)\sum_{t\in[r]}n_t=\sum_{l,j\in[r]}(n_l+n_j)>{r\choose 2}(n-k),$$
and then $r(n-k)<2n$.
Since $k<\frac{n}{2}$, this yields a contradiction.
Therefore, $r$ is equal to 2 or 3.
If $r=2$, then $|S_i|\geq n_1\cdot n_2\geq k(n-k)$.
If $r=3$, then there is an $n_l$ such that $k\leq n_l\leq n-k$, say $l=1$.
Otherwise, $n_j<k$ for each $j\in[3]$, then $n=\sum_{j\in[3]}n_j <n$, a contradiction.
Thus, $|S_i|> n_1\cdot (n_2+ n_3)\geq n(n-k)$.
\end{proof}
By the inequality (1) above, we have
\begin{align*}
\mu(G)&\leq \left\lfloor\frac{n+k-1}{k}\right\rfloor\cdot
\frac{n-1-\frac{k}{2}\left\lfloor\frac{n-1}{k}\right\rfloor}{n-1}=
\left\lfloor\frac{n+k-1}{k}\right\rfloor\cdot
\left(1-\frac{k}{2(n-1)}\left\lfloor\frac{n-1}{k}\right\rfloor\right)\\
&\leq \left(\frac{n+k-1}{k}\right)\cdot
\left[1-\frac{k}{2(n-1)}\left(\frac{n-1}{k}-1\right)\right]\\
&=\frac{n+k-1}{k}\cdot\frac{n+k-1}{2(n-1)}
< \frac{(n+k)^2}{2k(n-1)}.
\end{align*}
Since $\sum_{i,j}d(v_i,v_j)=\mu(G)\cdot{n\choose 2}$, we have
$\sum_{i,j}d(v_i,v_j)<\frac{(n+k)^2n}{4k}$.
It is obvious that $d(v_i,v_j)\geq c_\Gamma(v_i,v_j)$ for any two vertices $v_i,v_j$ of $G$.
Thus,
$$md(G)\leq \frac{\Sigma_{i\in \Gamma}|S_i|}{k(n-k)}= \frac{\sum_{i,j}c_\Gamma(v_i,v_j)}{k(n-k)}\leq\frac{\sum_{i,j}d(u,v)}{k(n-k)}
<\frac{(n+k)^2n}{4k^2(n-k)}
=\frac{(1+\varepsilon)^2}{4\varepsilon^2(1-\varepsilon)}.$$
The proof is thus complete.
\end{proof}

\begin{remark} Since $\varepsilon<\frac 1 2$, we have
$\frac{(1+\varepsilon)^2}{4\varepsilon^2(1-\varepsilon)}<(\frac 3 2)^2/2\varepsilon^2=\frac {9} {8\varepsilon^2}.$
This means that when the connectivity of a graph increases, its $MD$-number could decrease, and the upper bound is $4$
when $\varepsilon$ is getting to $\frac 1 2$.
\end{remark}

The following result gives a relationship between the monochromatic disconnection
number and the independent number of a graph.

\begin{theorem}
If $G$ is a $2$-connected graph, then $md(G)\leq \alpha(G)$.
The bound is sharp.
\end{theorem}
\begin{proof}
Let $P$ be a path and let $t\geq 2$ be an integer.
Since $\alpha(K_t\Box P)=|P|=md(K_t\Box P)$, the bound is sharp if the result holds.

The proof proceeds by induction on the order $n$ of a graph $G$.
If $n\leq 2\alpha(G)$, then since $G$ is a $2$-connected graph, $md(G)\leq \alpha(G)$.
If $G$ has a vertex $v$ such that $G-v$ is still $2$-connected, then by Lemma \ref{S-mono-1} (5),
we know $md(G-v)\geq md(G)$. Since $\alpha(G-v)\leq \alpha(G)$, by induction, we have $md(G)\leq md(G-v)\leq \alpha(G-v)\leq \alpha(G)$.
Thus, we only need to consider the graph $G$ with the property that $G-v$ is not a $2$-connected graph for any vertex $v$ of $G$.

Let $u$ be a vertex of $G$ such that $G-u$ has a maximum component.
Let $\mathcal{B}=\{D_1,\cdots,D_s\}$ be the set of components of $G-u$ and let $D_r$ be a maximum component.
Let $S$ be the set of cut-vertices of $G-u$.
The {\em block-tree} of $G-u$, denoted by $T$, is a bipartite graph with bipartition $\mathcal{B}$ and $S$,
and a block $D_i$ has an edge with a cut-vertex $v$ in $T$ if and only if $D_i$ contains $v$.
Then the leaves of $T$ are blocks, say $D_{k_1},\cdots,D_{k_l}$.
Since $G$ is $2$-connected, there is a vertex $v_i$ of $D_{k_i}-S$ such that $u$ connects $v_i$ in $G$ for $i\in[l]$.
We use $P_{i,j}$ to denote the subpath of $T$ from $D_{k_i}$ to $D_{k_j}$.
We now prove that $T$ is a path and $D_i$ is an edge for $i\neq r$.
If $T$ is not a path, then $l\geq 3$.
There are two leaves of $T$, say $D_{k_1}$ and $D_{k_2}$, such that $D_r\in V(P_{1,2})$.
Then $G-v_3$ has a component containing $V(D_r)\cup \{u\}$, which contradicts that $D_r$ is maximum.
Thus, $T$ is a tree.
Suppose $r\neq j$ and $D_j$ is not an edge, i.e., $D_j$ is a $2$-connected graph.
Since $T$ is a path, we have $W=V(D_j)-S-\{v_1,\cdots,v_l\}\neq \emptyset$.
Let $u'\in W$.
Then $G-u'$ has a component containing $V(D_r)\cup \{u\}$, which contradicts that $D_r$ is maximum.
Thus, $D_i$ is an edge for $i\neq r$.

Without loss of generality, suppose $V(D_i)\cap V(D_{i+1})=\{u_i\}$ for $i\in[s-1]$.
Then, $D_1,D_s$ are leaves of $T$, $D_i$ is an edge for $i\neq r$ and $S=\{u_1,\cdots,u_{s-1}\}$.
Let $u_0\in V(D_1-S)$ and $u_s\in V(D_s-S)$ be two vertices adjacent to $u$.

Let $P_1=\bigcup_{i<r}D_i$ and let $P_2=\bigcup_{i=r+1}^{s}D_i$.
Then $P_1$ and $P_2$ are paths.
There is an independent set $U_i$ of $P_i$ such that $U_i\cap V(D_r)=\emptyset$ and $|U_i|=\left\lceil\frac{|P_i|-1}{2}\right\rceil$ for $i\in[2]$.
Let $U$ be a maximum independent set of $D_r$.
Then $U\cup U_1\cup U_2$ is an independent set of $G-u$,
i.e.,
\begin{align*}
\alpha(G)&\geq \alpha(G-v)\geq|U\cup U_1\cup U_2|=\alpha(D_r)+\left\lceil\frac{|P_1|-1}{2}\right\rceil+
\left\lceil\frac{|P_2|-1}{2}\right\rceil\\
&\geq \alpha(D_r)+\left\lceil\frac{|P_1|+|P_2|-2}{2}\right\rceil=
\alpha(D_r)+\left\lceil\frac{s-1}{2}\right\rceil.
\end{align*}

Let $P=\{uu_0,uu_s\}\cup(\bigcup_{i\neq r}D_i)$ and let $G'=D_r\cup P$.
Then $P$ is an $(s+1)$-path and $G'$ is a $2$-connected spanning subgraph of $G$.
By Lemma \ref{S-mono-1} (3), we have $md(G)\leq md(G')$.
Let $\Gamma$ be an extremal $MD$-coloring of $G'$.
Then $\Gamma$ is an $MD$-coloring restricted on $D_r$ and $P$.
We call $D_r$ and each edge of $P$ the {\em joints} of $G'$.
Let $C$ be the set of colors $c\in\Gamma(G')$ such that $c$ is in at least two joints of $G'$.
For $c\in C$, we use $n_c$ to denote the number of joints of $G$ having edges colored with $c$.
Then $md(G')=|\Gamma(G')|=|\Gamma(D_r)|+||P||-\Sigma_{c\in C}(n_c-1)$.
Since there is a color $c$ of $C_\Gamma(u_{r-1},u_r)$ that separates $u_{r-1}$ and $u_r$, we have
$c\in \Gamma(D_r)\cap \Gamma(P)$.
By the same reason, for each $e\in E(P)$, either $\Gamma(e)=\Gamma(f)$ for an edge $f$ of $P-e$, or $\Gamma(e)\subseteq \Gamma(D_r)$.
Thus, $\Sigma_{c\in C}(n_c-1)\geq \left\lceil\frac{s+2}{2}\right\rceil$.
Therefore,
\begin{align*}
md(G)&\leq md(G')= |\Gamma(D_r)|+||P||-\Sigma_{c\in C}(n_c-1)\\
&\leq\alpha(D_r)+s+1-\left\lceil\frac{s+2}{2}\right\rceil
=\alpha(D_r)+ \left\lfloor\frac{s}{2}\right\rfloor\\
&=\alpha(D_r)+\left\lceil\frac{s-1}{2}\right\rceil\leq\alpha(G).
\end{align*}
The proof is thus complete.
\end{proof}

\section{Characterization of extremal graphs}

We knew that $md(G)=\left\lfloor\frac{n}{2}\right\rfloor$ if $G$ is a $2$-connected graph.
In this section, we characterize all the $2$-connected graphs with $MD$-number
$\left\lfloor\frac{n}{2}\right\rfloor$. We use $\mathcal{E}=(L_0;L_1,\cdots,L_t)$
to denote an ear-decomposition of $G$, where $L_0$ is a $2$-connected subgraph of $G$
and $L_i$ is a path for $i\in[t]$. Let $Z_\mathcal{E}=\{L_i \ | \ i>0\mbox{ and }end(L_i)\subseteq V(L_0)\}$.

If $C$ is a cycle of $G$ and $v\in V(G)-V(C)$,
then we use $\kappa(v,C)$ to denote the maximum number of $vv_i$-path $P_i$ of $G$,
such that $V(P_i)\cap V(P_j)=\{v\}$ and $V(P_i)\cap V(C)=\{v_i\}$.
We call $H=C\cup (\bigcup_{i=1}^{\kappa(v,C)}P_i)$ a
{\em $(v,C)$-umbrella} of $G$ (or an {\em umbrella} for short) if $\kappa(v,C)\geq3$.
The vertices $v_1,\cdots,v_{\kappa(v,C)}$ divide $C$ into $\kappa(v,C)$ paths, say
$P'_1,\cdots,P'_{\kappa(v,C)}$.
We call $P_i$ a {\em spoke} of $H$ and call $P'_i$ a {\em rim} of $H$.
If the size of each spoke is odd and the size of each rim is even, then we call
the $(v,C)$-umbrella a {\em uniform $(v,C)$-umbrella}
(or {\em uniform umbrella} for short).

A graph $G$ is called a {\em $\theta$-graph}
if $G$ is the union of three
internal disjoint paths $T_1,T_2$ and $T_3$ with
$end(T_1)=end(T_2)=end(T_3)$.
If each $T_i$ is an even path, then we call $G$ an {\em even $\theta$-graph}
and call each $T_i$ a {\em route}.

Suppose $\mathcal{E}=(L_0;L_1,\cdots L_t)$ is an ear-decomposition of $G$.
Then the concept {\em normal ear-decomposition} of $G$ is defined as follows.

$\bullet$ If $|G|$ is even, then $\mathcal{E}$ is a
normal ear-decomposition of $G$ if $L_0$ is a cycle.

$\bullet$ If $|G|$ is odd and $G$ is not a bipartite graph,
then $\mathcal{E}$ is a normal ear-decomposition of $G$ if $L_0$ is an odd cycle.

$\bullet$ If $|G|$ is odd and $G$ is a bipartite graph, then $\mathcal{E}$ is a
normal ear-decomposition of $G$ if $L_0$ is either an umbrella
or an even $\theta$-graph. Moreover, if $L_0$ is an even $\theta$-graph,
then for each $L_i\in Z_\mathcal{E}$, $end(L_i)$ is contained in one route.

\begin{lemma}\label{S-mono-normal}
If $G$ is a $2$-connected graph, then $G$ has a normal ear-decomposition.
\end{lemma}
\begin{proof}
If $n$ is even or $G$ is a nonbipartite graph with $n$ odd, then $G$
has a normal ear-decomposition.
If $G$ is a bipartite graph and $n$ is odd,
then let $\mathcal{E}=\{L_0;L_1,\cdots,L_t\}$
be an ear-decomposition of $G$
with $L_0$ an even cycle.
Since $n=|L_0|+\Sigma_{i\in [t]}(|L_i|-2)$ and $n$ is odd, there is an even path among the ears, say $L_i$.
Since $H=\bigcup_{l=0}^{i-1}L_i$ is a $2$-connected bipartite graph,
there is an even cycle $C$ of $H$ containing $end(L_i)$.
Moreover, $end(L_i)$ divides $C$ into two even paths.
So, $L'_0=C\cup L_i$ is an even $\theta$-graph, say the three routes are $T_1,T_2$ and $T_3$.
Let $\mathcal{E}'=\{L'_0;L'_1,\cdots,L'_s\}$
be an ear-decomposition of $G$ and let $end(L'_j)=\{u_i,v_i\}$ for $j\in[s]$.
If the ends of each $L'_j$ in $Z_{\mathcal{E}'}$ are contained in one route, then $\mathcal{E}'$ is a
normal ear-decomposition of $G$.
Otherwise, suppose $L'_j\in Z_{\mathcal{E}'}$,
$u_j\in I(T_1)$ and $v_j\in I(T_2)$.
Then $\kappa(u_j, T_2\cup T_3)\geq 3$,
i.e., there is a $(u_j, T_2\cup T_3)$-umbrella, say $M$.
Then there is a normal ear-decomposition of $G$ containing $M$.
\end{proof}

\begin{lemma} \label{2-con-eard}
Suppose $G$ is a $2$-connected graph with
$md(G)=\left\lfloor\frac{n}{2}\right\rfloor$.
Let $\mathcal{E}=(L_0;L_1,\cdots,L_t)$ be an ear-decomposition of $G$ with
$L_0$ a $2$-connected subgraph of $G$ and $end(L_i)=\{a_i,b_i\}$ for $i\in[t]$.
Then we have the following results.
\begin{enumerate}
\item If $H$ is a $2$-connected subgraph of $G$,
then each extremal $MD$-coloring of $G$ is an extremal $MD$-coloring restricted on $H$, and $md(H)=\left\lfloor\frac{|H|}{2}\right\rfloor.$
\item If $n$ is even, then $G$ is a bipartite graph
  and $L_i$ is an odd path for $i\in[t]$.
\item If $n$ is odd, then
 when $|L_0|$ is even, exact one of $\{||L_1||,\cdots,||L_t||\}$ is even; when $|L_0|$ is odd, $L_i$ is an odd path for $i\in[t]$.
\end{enumerate}
\end{lemma}
\begin{proof}
Let $\Gamma$ be an extremal $MD$-coloring of $G$.
Then for each $i\in[t]$,
$\Gamma(L_i)\cap \Gamma(\bigcup_{l=0}^{i-1}L_l)\neq \emptyset$;
otherwise, $C_\Gamma(a_i,b_i)=\emptyset$, a contradiction.
Moreover, each color of $\Gamma(L_i)-\Gamma(\bigcup_{l=0}^{i-1}L_l)$
is used on at least two edges of $L_i$.
Otherwise, suppose $p\in\Gamma(L_i)-\Gamma(\bigcup_{l=0}^{i-1}L_l)$ and color $p$ is only used on one edge $e=xy$ of $L_i$.
Then since $\Gamma(\bigcup_{l=0}^{i}L_l)-e$ is connected,
$C_\Gamma(x,y)=\emptyset$, a contradiction.
Therefore,
\begin{align*}
\left\lfloor\frac{n}{2}\right\rfloor&=md(G)
=|\Gamma(L_0)|+\Sigma_{i=1}^t|\Gamma(L_i)-\Gamma(\bigcup_{l=0}^{i-1}L_l)|\\
&\leq  md(L_0)+\Sigma_{i=1}^t
\left\lfloor\frac{||L_i||-1}{2}\right\rfloor\\
&\leq \left\lfloor\frac{|L_0|}{2}\right\rfloor+\Sigma_{i=1}^t
\left\lfloor\frac{||L_i||-1}{2}\right\rfloor\\
&\leq\left\lfloor\frac{|L_0|}{2}+\Sigma_{i\in[t]}
\frac{||L_i||-1}{2}\right\rfloor
=\left\lfloor\frac{n}{2}\right\rfloor.
\end{align*}
Then $|\Gamma(L_0)|=md(L_0)=\left\lfloor\frac{|L_0|}{2}\right\rfloor$ and
$|\Gamma(L_i)|=\left\lfloor\frac{||L_i||-1}{2}\right\rfloor$ for each $i\in[t]$.
So, $\Gamma$ is an extremal $MD$-coloring restricted on $L_0$, and
$md(L_0)=\left\lfloor\frac{|L_0|}{2}\right\rfloor$.
Moreover, $|\Gamma(L_i)\cap\Gamma(\bigcup_{l=0}^{i-1}L_l)|=1$
when $L_i$ is an odd path.

If $G$ is not a bipartite graph, $n$ is even
and $L_0$ an odd cycle,
then the above inequality does not hold.
Thus, $G$ is a bipartite graph when $n$ is even.
Moreover, $L_i$ is an odd path for each $i\in[t]$.
If $n$ and $|L_0|$ are odd,
then $L_i$ is an odd path for $i\in[t]$.
If $n$ is odd and $|L_0|$ is even, then exact one of
$\{||L_1||,\cdots,||L_t||\}$ is even.
\end{proof}

For a normal ear-decomposition
$\mathcal{E}=\{L_0;L_1,\cdots,L_t\}$ of a $2$-connected graph $G$,
if $L_0$ is an odd cycle and $L_i\in Z_\mathcal{E}$,
then $end(L_i)$ divides $L_0$ into an odd path and an even path,
which are denoted by $f_o(\mathcal{E},i)$ and $f_e(\mathcal{E},i)$, respectively. If $L_0$ is an even cycle, $L_i\in Z_\mathcal{E}$ and $e\in E(L_0)$, then we use $g(\mathcal{E},i,e)$ to denote the
subpath of $L_0$ with ends $end(L_i)$ and $g(\mathcal{E},i,e)$ contains $e$. We define a function $f(\mathcal{E},i,j)$ for $0\leq i<j\leq t$ as follows.
\begin{equation*}
f(\mathcal{E},i,j)=
\begin{cases}
f_o(\mathcal{E},j)& i=0,L_j\in Z_\mathcal{E}\mbox{ and }L_0
                 \mbox{ is an odd cycle};\\
g(\mathcal{E},i,e)& i=0,L_j\in Z_\mathcal{E}\mbox{ and }L_0
                 \mbox{ is an even cycle with }e\in E(L_0);\\
a_jPb_j& i=0,L_j\in Z_\mathcal{E},L_0\mbox{ is an umbrella},
            P\mbox{ is either a spoke}\mbox{ or a rim of }\\
            &L_0\mbox{ such that }end(L_j)\subseteq V(P);\\
a_jTb_j& i=0,L_j\in Z_\mathcal{E},L_0\mbox{ is an even }
          \theta\mbox{-graph},\ T\mbox{ is one of the three }\\
            &\mbox{routes such that }end(L_i)\subseteq V(T);\\
a_jL_ib_j& i>0\mbox{ and }end(L_j)\subseteq V(L_i);\\
K_4& otherwise.
\end{cases}
\end{equation*}
If $L_0$ is not an even cycle, then the function depends only on $\mathcal{E},i$ and $j$.
If $L_0$ is an even cycle and $i=0$, then the function also depends on $e$.
Thus, we need to fix an edge $e$ of $L_0$ in advance if $L_0$ is an even cycle.
\begin{lemma}\label{S-mono-ex}
If $G$ is a uniform umbrella or an even $\theta$-graph other than $K_{2,3}$, then
$|G|$ is odd and $md(G)=\left\lfloor\frac{|G|}{2}\right\rfloor$.
\end{lemma}
\begin{proof}
It is obvious that $|G|$ is odd.
Fix an integer $k\geq 3$.
Suppose $G'$ is either a minimum even $\theta$-graph other than $K_{2,3}$,
or a minimum uniform umbrella with $k$ spokes.

If $G'$ is a minimum even $\theta$-graph other than $K_{2,3}$,
then $G'$ and one of its extremal $MD$-colorings are depicted in Figure \ref{S-mono-umbre-base} (1),
which implies $md(G')=3= \left\lfloor\frac{|G'|}{2}\right\rfloor$.

If $G'$ is a minimum uniform umbrella with $k$ spokes,
then each spoke is an edge and each rim is a $2$-path.
Suppose the $k$ spokes are $e_1=vv_1,\cdots,e_k=vv_k$, and the $k$ rims are
$P_1=v_1f_1u_1f_2v_2,\cdots,P_k=v_kf_{2k-1}u_kf_{2k}v_1$.
We color each $e_i$ with $i$.
The colors of the edges of $P_i$ obey the rule that
opposite edges of any $4$-cycle have the same color (see Figure
\ref{S-mono-umbre-base}).
\begin{figure}[h]
    \centering
    \includegraphics[width=250pt]{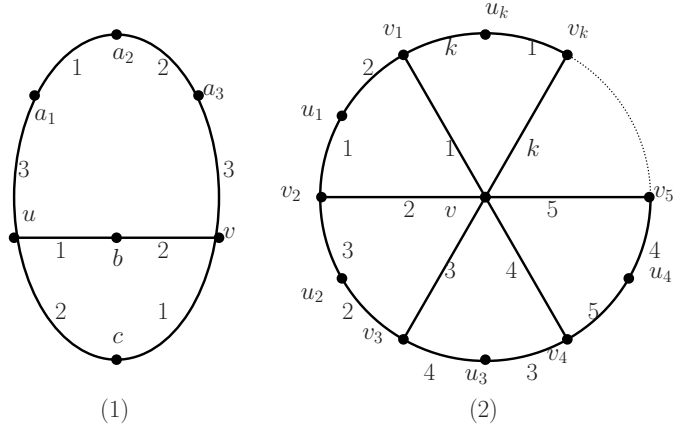}\\
    \caption{Extremal $MD$-colorings of the minimum even $\theta$-graph and the minimum uniform umbrella.} \label{S-mono-umbre-base}
\end{figure}
Since $k\geq 3$, we know that for $v_1$, $\{e_1,f_2,f_{2k-1}\}$ is a monochromatic $v_1v$-cut (it is also a monochromatic $v_1v_i$-cut for $i\neq 1$, and a monochromatic $v_1u_i$-cut for $i\neq \{1,2,k\}$), $\{e_2,f_1,f_4\}$ is a monochromatic $v_1u_1$-cut and
$\{e_k,f_{2k},f_{2k-3}\}$ is a monochromatic $v_1u_k$-cut.
By symmetry, the edge-coloring is an $MD$-coloring of $G'$ with $k$ colors.
Since $G'$ is $2$-connected and $|G'|=2k+1$, we have $md(G')=k=
\left\lfloor\frac{|G'|}{2}\right\rfloor$.

Suppose $G$ is a uniform umbrella with $k$ spokes
(an even $\theta$-graph other than $K_{2,3}$).
Then $G$ is obtained from
$G'$ by replacing some edges with odd paths, respectively.
W.l.o.g., suppose $G$ is obtained from $G'$ by replacing one edge with
an odd path $P$.
Then by Lemma \ref{S-mono-path}, we have $md(G)\geq md(G')+
\left\lfloor\frac{||P||-1}{2}\right\rfloor=
\left\lfloor\frac{|G|}{2}\right\rfloor$, i.e.,
$md(G)=\left\lfloor\frac{|G|}{2}\right\rfloor$.
The proof is thus complete.
\end{proof}

\begin{lemma}\label{S-mono-uniform}
If $G$ is a bipartite graph of odd order and
$md(G)=\left\lfloor\frac{n}{2}\right\rfloor$,
then each umbrella of $G$ is a uniform umbrella.
\end{lemma}
\begin{proof}
Suppose $G$ is a bipartite graph of odd order and
$md(G)=\left\lfloor\frac{n}{2}\right\rfloor$.
Let $H$ be a $(v,C)$-umbrella of $G$.
We show that $H$ is a uniform umbrella.

If $\kappa(v,C)=3$, then let $R_1,R_2$ and $R_3$ be spokes of $H$
and $R_i$ be a $vv_i$-path.
Then $C$ is divided into three paths by vertices $v_1,v_2$ and $v_3$
(say, the three paths are $W_1,W_2$ and $W_3$, such that
$end(W_1)=\{v_1,v_2\}$, $end(W_2)=\{v_2,v_3\}$ and $end(W_3)=\{v_1,v_3\}$).
If each $R_i$ is an odd path, then since $G$ is a bipartite graph,
each $W_i$ is an even path, $H$ be a uniform $(v,C)$-umbrella of $G$.
If, by symmetry, $R_1$ is an even path and $R_2,R_3$ are odd paths, then $W_1,W_3$ are odd paths
and $W_2$ is an even path.
Then since $(W_1\cup W_3\cup R_2\cup R_3; R_1,W_2)$
is an ear-decomposition of $H$ containing even paths $R_1$ and $W_2$,
by Lemma \ref{2-con-eard} (1) and (3) this yields a contradiction.
If, by symmetry, $R_1$ is an odd path and $R_2,R_3$ are even paths,
then $H$ is a uniform $(v_1,R_2\cup R_3\cup W_2)$-umbrella.
If each $R_i$ is an even path, then $(C;R_1\cup R_2,R_3)$ is an
ear-decomposition of $H$ containing two even paths, a contradiction.

If $\kappa(v,C)\geq4$, then let $W_1,W_2,W_3,W_4$ be four spokes of
$H$ (let $W_i$ be a $vv_i$ path for $i\in[4]$).
Then $C$ is divided into two paths by $v_2$ and $v_3$
(say, the two paths are $Y_1$ and $Y_2$).
W.l.o.g., suppose $W_1$ is an even path. Then $(Y_1\cup W_2\cup W_3;Y_2,W_4,W_1)$
is an ear-decomposition of $H$.
Since $md(H)=\left\lfloor\frac{|H|}{2}\right\rfloor$
and $W_1$ is an even path, by Lemma \ref{2-con-eard} (4),
$Y_2$ is an odd path.
Since $H$ is a bipartite graph, either $W_2$ or $W_3$ is an even path (say $W_2$).
Then $(C\cup W_3\cup W_4;W_1,W_2)$
is an ear-decomposition of $H$ containing two even paths, a contradiction.
So, each spoke of $H$ is an odd path.
Since $H$ is a bipartite graph, each rim of $H$ is an even path.
\end{proof}

Suppose $\mathcal{E}=(L_0;L_1,\cdots L_t)$ is an ear-decomposition of $G$.
Then $\mathcal{E}$ can have the following possible properties.

{\bf Q}: If $end(L_j)\cap I(L_i)\neq \emptyset$,
then $end(L_j)\subseteq V(L_i)$.

{\bf R}: If $end(L_j)\cap I(f(\mathcal{E},k,i))\neq \emptyset$,
then $f(\mathcal{E},k,j)$ is a proper subpath of $f(\mathcal{E},k,i)$.

The concept {\em standard ear-decomposition} of $G$
is defined as follows.

$\bullet$ If $|G|$ is even, then $\mathcal{E}$ is a
standard ear-decomposition of $G$ if $L_0$ is an even cycle.

$\bullet$ If $|G|$ is odd and $G$ is not a bipartite graph,
then $\mathcal{E}$ is a
standard ear-decomposition of $G$ if $L_0$ is an odd cycle and
$f_e(\mathcal{E},i)\cap f_e(\mathcal{E},j)\neq \emptyset$ for $L_i,L_j\in Z_\mathcal{E}$.

$\bullet$ If $|G|$ is odd and $G$ is a bipartite graph, then $\mathcal{E}$ is a
standard ear-decomposition of $G$ if $L_0$ is either a uniform umbrella
or a even $\theta$-graph other than $K_{2,3}$.
Moreover, for each $L_i\in Z_\mathcal{E}$, if $L_0$
is a uniform umbrella, then $end(L_i)$ is contained in either a rim or a spoke;
if $L_0$ is an even $\theta$-graph other than $K_{2,3}$, then $end(L_i)$ is contained in one route.

Therefore, a standard ear-decomposition of $G$ is also a
normal ear-decomposition of $G$.

\begin{lemma}\label{S-mono-poset}
If $\mathcal{E}=(L_0;L_1,\cdots,L_t)$ is a standard ear-decomposition of $G$ and
$\mathcal{E}$ has properties {\bf Q} and {\bf R}, then
there exist integers $0\leq k<r\leq t$ such that $end(L_r)\subseteq V(L_k)$, and $d(u)=2$ for each $u\in I(f(\mathcal{E},k,r))\cup I(L_r)$.
\end{lemma}
\begin{proof}
For $i\in[t]$, let $end(L_i)=\{a_i,b_i\}$.
We use $m_r$ ($n_r$) to demote the minimum integer
such that $a_r\in V(L_{m_r})$ ($b_r\in V(L_{n_r})$).
Since $I(L_0)=V(L_0)$, we have
$a_i\in I(L_{m_r})$ and $b_r\in I(L_{n_r})$.
Since $\mathcal{E}$ has property {\bf Q}, we know
for each $i\in[t]$, either $end(L_i)\subseteq V(L_{m_i})$,
or $end(L_i)\subseteq V(L_{n_i})$.
Let $l_i$ be the minimum integer such that $end(L_i)\subseteq V(L_{l_i})$.

Let $D$ be a digraph with vertex-set $V(D)=\{s_0,s_1,\cdots,s_t\}$
and arc-set $A(D)=\{(s_i,s_j) \ | \ f(\mathcal{E},i,j)\neq K_4\}$.
We use $d_j$ to denote the length of a minimum directed path from $s_0$ to $s_j$. If $end(L_j)\cap I(L_i)\neq \emptyset$, then $d_j=d_i+1$.
Let $U=\{j \ | \ d_j\mbox{ is maximum}\}$.
If $j\in U $, then $d_G(u)=2$ for each $u\in I(L_j)$.

Let $i$ be an integer in $U$ such that $|f(\mathcal{E},l_i,i)|$ is minimum.
If there is a vertex $v$ of $I(f(\mathcal{E},l_i,i))$ such that $d_G(v)\geq3$,
then there is a path $L_k$ such that $v\in end(L_k)\cap I(f(\mathcal{E},l_i,i))$.
Since $\mathcal{E}$ has property {\bf R},
$f(\mathcal{E},l_i,k)$ is a proper subpath of $f(\mathcal{E},l_i,i)$, i.e.,
$|f(\mathcal{E},l_i,k)|<|f(\mathcal{E},l_i,i)|$.
Since $|f(\mathcal{E},l_i,i)|$ is minimum, we have $k\notin U$.
Then there is a path, say $L_p$, such that $end(L_p)\cap I(L_k)\neq \emptyset$. Thus, $d_p>d_k=d_i$, a contradiction.
Hence, $d_G(u)=2$ for each $u\in I(f(\mathcal{E},l_i,i))$.
\end{proof}

\begin{theorem}\label{S-mono-mdex}
Suppose $G$ is a $2$-connected graph and
$\mathcal{E}=(L_0;L_1,\cdots L_t)$ is a normal ear-decomposition of $G$.
Then $md(G)=\left\lfloor\frac{n}{2}\right\rfloor$
if and only if
$\mathcal{E}$ is a standard ear-decomposition of $G$
that has properties {\bf Q} and {\bf R},
$L_i$ is an odd path for each $i\in[t]$,
and $f(\mathcal{E},i,j)$ is an odd path
if $f(\mathcal{E},i,j)\neq K_4$.
\end{theorem}
\begin{proof}
For $i\in[t]$, let $end(L_i)=\{a_i,b_i\}$.

For the necessity, suppose $md(G)=\left\lfloor\frac{n}{2}\right\rfloor$.
If $n$ is even, then $L_0$ is an even cycle.
By Lemma \ref{2-con-eard} (2), $G$ is a bipartite graph and $L_i$ is an odd path for $i\in[t]$.
Since $f(\mathcal{E},i,j)\cup L_j$ is an even cycle, $f(\mathcal{E},i,j)$
is an odd path.
If $n$ is odd,
then since $\mathcal{E}$ is normal, $|L_0|$ is odd.
By Lemma \ref{2-con-eard} (4),
$L_i$ is an odd path for $i\in[t]$.
Suppose there are integers $i,j$ such that $f(\mathcal{E},i,j)$ is an even path.
If $i=0$ and $L_0$ is an odd cycle,
then $f(\mathcal{E},i,j)=f_o(i,j)$ is an odd path, a contradiction.
If $i>0$ and $L_0$ is an odd cycle, then
$H=L_j\cup(\bigcup_{c=0}^iL_c)$ is a $2$-connected subgraph of $G$
and $(L_0;L_1\cdots,L_{i-1},L_i\cup L_j-I(f(\mathcal{E},i,j)),f(\mathcal{E},i,j))$
is an ear-decomposition of $H$ with $L_0$ an odd cycle and $f(\mathcal{E},i,j)$
an even path, and by Lemma \ref{2-con-eard} (1) and (3) this yields a contradiction.
If $L_0$ is an umbrella or an even $\theta$-graph other than $K_{2,3}$,
then $G$ is a bipartite graph.
Since $f(\mathcal{E},i,j)\cup L_j$ is an even cycle and $L_j$ is an odd path, $f(\mathcal{E},i,j)$
is an odd path, a contradiction.
Thus, $f(\mathcal{E},i,j)$ is an odd path if $n$ is odd.

We need to prove that $\mathcal{E}$ is standard
and $\mathcal{E}$ has properties {\bf Q} and {\bf R} below.

\begin{claim}
$\mathcal{E}$ is standard.
\end{claim}
\begin{proof}
If $n$ is even, then since $G$ is a bipartite graph, $L_0$ is an even cycle. Thus, $\mathcal{E}$ is standard.

If $G$ is not a bipartite graph and $n$ is odd, then $L_0$ is an odd cycle. Suppose $\mathcal{E}$ is not a standard ear-decomposition of $G$.
Then there are paths $L_i$ and $L_j$ of $Z_\mathcal{E}$ such that
$E(f_e(\mathcal{E},i))\cap E(f_e(\mathcal{E},j))= \emptyset$.
Let $D=L_i\cup L_j\cup[L_0-I(f_e(\mathcal{E},i)\cup f_e(\mathcal{E},j))]$.
Then $D$ is $2$-connected subgraph of $L_0\cup L_j\cup L_i$.
Since $(D;f_e(\mathcal{E},i), f_e(\mathcal{E},j))$ is an ear-decomposition of
$L_0\cup L_i\cup L_j$ and $f_e(\mathcal{E},i), f_e(\mathcal{E},j)$ are even paths, by Lemma \ref{2-con-eard} (1) and (3) this yields a contradiction.
Thus, $\mathcal{E}$ is standard.

If $G$ is a bipartite graph, $n$ is odd and $L_0$ is
an even $\theta$-graph,
then $L_0\neq K_{2,3}$.
Otherwise $L_0$ is a $2$-connected subgraph of $G$ with
$md(L_0)=1<\left\lfloor\frac{|L_0|}{2}\right\rfloor$,
and by Lemma \ref{2-con-eard} (1) this yields a contradiction.
Thus, $\mathcal{E}$ is standard.

If $G$ is a bipartite graph, $n$ is odd and $L_0$ is an umbrella,
then suppose the rims of $L_0$ are
$W_1,\cdots,W_k$, where $k\geq 3$ and
$W_i$ is a $v_iv_{i+1}$-path for $i\in[k-1]$.
Suppose the spokes are $R_1,\cdots,R_k$, where $R_i$ is a $vv_i$-path.
Let $C=\bigcup_{i\in[k]} W_i$.
Since $md(G)=\left\lfloor\frac{n}{2}\right\rfloor$,
by Lemma \ref{S-mono-uniform}, $L_0$ is a uniform umbrella, i.e.,
each $W_i$ is an even path and each $R_i$ is an odd path.
Suppose there is a path $L_i$ of $Z_\mathcal{E}$ such that
$end(L_i)$ is neither contained in any spoke nor contained in any rim.
If $a_i\in I(R_j)$ and $b_i\in V(L_0)-V(R_j)$,
then $a_i$ divides $R_j$ into two subpaths
$R^1_j=vL_ja_i$ and $R^2_j=a_iL_jv_j$.
Since $k\geq 3$, w.l.o.g., let $b_j\notin I(W_k)$.
Then $H_s=R^s_j\cup L_i\cup(\bigcup_{l\neq k}W_l)\cup (\bigcup_{l\neq j}R_l)$ is a $2$-connected graph for $s\in[2]$.
Since $L_j$ is an odd path, one of $R^1_j$ and $R^2_j$ is an even path,  say $R^1_j$.
Since $(H_2;W_k,R^1_j)$ is an ear-decomposition of $L_0\cup L_i$ and $W_k,R^1_j$ are even paths, by Lemma \ref{2-con-eard} (1) and (3) this yields a contradiction.
If $end(L_i)\subseteq V(C)$, then since $G$ is a bipartite graph,
$L_i$ is an odd path and each $W_j$ is an even path,
we have $|end(L_i)\cap \{v_1,\cdots,v_k\}|\leq1$.
Therefore, there is a rim $W_j$ such that
$a_i$ divides $W_j$ into two odd paths
$W^1_j=v_jW_ja_i$ and $W^2_j=a_iW_jv_{j+1}$.
(w.l.o.g., suppose $1\leq j<k$).
Since there is no rim containing $end(L_i)$, we have $b_i\notin V(W_j)$.
Note that $end(L_i)$ divides $C$ into two subpaths $C^1$ and $C^2$ such that $v_j\in V(C^1)$ and $v_{j+1}\in V(C^2)$.
Since $k\geq 3$, by symmetry, suppose $|C^1\cap \{v_1,\cdots, v_k\}|\geq 2$.
Then there is an integer $l\in[k]-\{i+1\}$ such that $C^1$ contains $v_i$ and $v_l$.
Then there is an ear-decomposition $(C';P'_1,P'_2,\cdots)$ of
$L_0\cup L_i$ such that
$C'=C^1\cup L_i, P'_1=R_i\cup R_l$ and $P'_2=W^2_i\cup R_{i+1}$.
Since $P'_1$ and $P'_2$ are even paths,
by Lemma \ref{2-con-eard} (3) this yields a contradiction.
Thus $\mathcal{E}$ is standard.
\end{proof}

\begin{claim}
$\mathcal{E}$ has property {\bf Q}.
\end{claim}
\begin{proof}
Let $m_i$ ($n_i$) be the minimum integer such that $a_i\in V(L_{m_i})$
($b_i\in V(L_{n_i})$). Since $I(L_0)=V(L_0)$, we have $a_i\in I(L_{m_i})$ and $b_i\in I(L_{n_i})$. Let $l_i$ be an integer such that $end(L_i)\cap I(L_{l_i})\neq \emptyset$.

Suppose $\mathcal{E}$ does not have property {\bf Q}.
Then there are integers $0\leq j<r\leq t$
such that $a_r\in I(L_j)$ and $b_r\notin V(L_j)$.
Since $b_r\in I(L_{n_r})$, by symmetry, suppose $j>l_{n_r}$.
For convenience, let $l_{n_r}=i$.
Since $L_j$ is an odd path, let $a_jL_ja_r$ be an even path.
Let $l=\max\{m_j,n_j,n_r\}$ and $H=L_j\cup L_r\cup(\bigcup_{h=0}^lL_h)$.
Then $H$ is a $2$-connected graph with an ear-decomposition
$(L_0;L_1,\cdots,L_l,a_rL_jb_j\cup L_r,a_jL_ja_r)$.
If $L_0$ is an odd cycle, or a uniform umbrella, or an even $\theta$-graph other than $K_{2,3}$,
then since $|L_0|$ is odd and $a_jL_ja_r$ is an even path,
by Lemma \ref{2-con-eard} (1) and (3) this yields a contradiction.
If $L_0$ is an even cycle, then by Lemma \ref{2-con-eard} (1) and (2)  this yields a contradiction.
\end{proof}

\begin{claim}
$\mathcal{E}$ has property {\bf R}.
\end{claim}
\begin{proof}
If $\mathcal{E}$ does not have property {\bf R},
then there are integers $r,i,j$
such that $end(L_j)\cap I(f(\mathcal{E},r,i))\neq \emptyset$
and $f(\mathcal{E},r,j)$ is not a subpath of $f(\mathcal{E},r,i)$.
Since $\mathcal{E}$ has property {\bf Q},
$f(\mathcal{E},r,j)$ is a subpath of $L_r$.
Then $end(L_i)$ and $end(L_j)$ appear
alternately on $L=f(\mathcal{E},r,i)\cup f(\mathcal{E},r,j)$, say $a_i,a_j,b_i,b_j$ are consecutively on $L$.
Here, $L$ is a subpath of the path $L_r$ if $r>0$;
$L$ is a subpath of either a rim or a spoke of $L_r$ if $r=0$ and $L_0$ is a uniform umbrella;
$L$ is a subpath of a route if $r=0$ and $L_0$ is an even $\theta$-graph other than $K_{2,3}$;
$L$ is a subpath of a cycle $L_r$ if $r=0$ and $L_0$ is a cycle.
Let $W^1=a_iLa_j,W^2=a_jLb_i$ and $W^3=b_iLb_j$.
Since $f(\mathcal{E},r,i)$ and $f(\mathcal{E},r,j)$ are odd paths,
either $W^1,W^3$ are even paths and $W^2$ is
an odd path, or $W^2$ is an even path and $W^1,W^3$ are odd paths.
Let $H=(\bigcup_{l=0}^rL_l)\cup L_i\cup L_j$.

Suppose $W^1,W^3$ are even paths and $W^2$ is an odd path.
Let $H'$ be a graph obtained from $H$ by removing $W^1$ and $W^3$.
Then $H'$ is a $2$-connected graph.
Since $(H';W^1,W^3)$ is an ear-decomposition of $H$ and $W^1,W^3$ are even paths, by Lemma \ref{2-con-eard} this yields a contradiction.

Suppose $W^2$ is an even path and $W^1,W^3$ are odd paths.
Let $H_i$ be a graph obtain from $H$
by removing $W^i$ for $i\in[3]$.
It is obvious that each $H_i$ is a $2$-connected graph.
If $L_0$ is an even cycle, then $(H_2;W^2)$ is an ear-decomposition of $G$, and by Lemma \ref{2-con-eard} (1) and (2) this yields a contradiction.
If $r=0$ and $L_0$ is an odd cycle, then $P=L_0-I(L)$ is an even path and $C=H_2-I(P)$ is an even cycle.
Since $(C;P,W^2)$ is an ear-decomposition of $H$ and $P,W^2$ are even paths, by Lemma \ref{2-con-eard} (1) and (3) this yields a contradiction.
If $r=0$ and $L_0$ is an even $\theta$-graph, then suppose $T_1,T_2$ and $T_3$ are routes of $L_0$, and suppose $L$ is a subpath of $T_1$.
Then $(H_2-I(T_2);T_2,W^2)$ is an ear-decomposition of $H$ and $T_2,W^2$ are even paths, a contradiction.
If $r=0$ and $L_0$ is a uniform umbrella, then there is a rim $W$ of $L_0$ such that $L$ is not a subpath of $W$.
Then $(H_2-I(W);W,W^2)$ is an ear-decomposition of $H$ and $W,W^2$ are even paths, a contradiction.
If $r>0$ and $n$ is odd, then $(L_0;\cdots,W^2)$ is an ear-decomposition of $H$.
Since $|L_0|$ is odd and $W^2$ is an even path,
by Lemma \ref{2-con-eard} (1) and (3) this yields a contradiction.
\end{proof}

Now for the sufficiency, suppose $\mathcal{E}=(L_0;L_1,\cdots,L_t)$ satisfies all conditions of the theorem, i.e.,
$\mathcal{E}$ is a standard ear-decomposition of $G$ that has properties {\bf Q} and {\bf R},
$L_i$ is an odd path for $i\in [t]$, and $f(\mathcal{E},j,i)$ is an odd path when $f(\mathcal{E},j,i)\neq K_4$.
Recall the definitions of digraph $D$, set $U$ and integer
$l_i$ in Lemma \ref{S-mono-poset}.
We choose an integer $r$ from $U$ such that $|f(\mathcal{E},l_r,r)|$ is minimum. For convenience, let $l=l_r$.
Then for each vertex $u$ of $I(f(\mathcal{E},l,r))\cup I(L_r)$,
we have $d_G(u)=2$.
The proof proceeds by induction on $t$. By Lemmas \ref{S-mono-1} (2) and \ref{S-mono-ex}, the result holds for $t=0$.

If $L_r$ is not an edge, then let $G'$ be a graph obtained from $G$ by replacing $f(\mathcal{E},l,r)$ with an edge $f=a_rb_r$,
let $G'_1=G'-I(L_r)$ and $G'_2=L_r\cup f$.
Let $L=[L_l-I(f(\mathcal{E},l,r))-E(f(\mathcal{E},l,r))]\cup f$. Let $\mathcal{E}'$ be an ear-decomposition of $G'_1$ obtained from $\mathcal{E}$ by removing $L_r$, and then replacing $L_l$ with $L$.
If $l>0$, then since $f(\mathcal{E},l,r)$ is an odd path, $L$ is an odd path and $\mathcal{E}'$ satisfies all the conditions.
If $l=0$ and $L_l$ is a uniform umbrella (an odd cycle or an even cycle), then $L$ is also a uniform umbrella (an odd cycle, an even cycle),
i.e., $\mathcal{E}'$ satisfies all the conditions in this case.
If $l=0$ and $L_l$ is an even $\theta$-graph,
then $\mathcal{E}'$ satisfies all the conditions except for $L=K_{2,3}$.
Thus, $\mathcal{E}'$ satisfies all the conditions unless $L=K_{2,3}$.

If $L\neq K_{2,3}$, then $\mathcal{E}'$ satisfies all the conditions.
Since the number of paths in $\mathcal{E}'$ is $t-1$, by the induction hypothesis we have $md(G'_1)=\left\lfloor\frac{|G'_1|}{2}\right\rfloor$.
Since $G'_2$ is an even cycle, we have $md(G'_2)=\frac{|G'_2|}{2}$.
Thus, by Lemma \ref{2-con-md1}, $md(G')=md(G'_1)+md(G'_2)-1=\left\lfloor\frac{|G'|}{2}\right\rfloor$.
Since $G$ is a graph obtained from $G'$ by replacing $f$ with the odd path $f(\mathcal{E},l,r)$, by Lemma \ref{S-mono-path} we have
$md(G)\geq md(G')+\left\lfloor\frac{||f(\mathcal{E},l,r)||-1}{2}\right\rfloor
=\left\lfloor\frac{n}{2}\right\rfloor$.
Therefore, $md(G)=\left\lfloor\frac{n}{2}\right\rfloor$.

If $L=K_{2,3}$, then $l=0$ and $r=1$.
Since $r\in U$, $d_r$ is maximum and $d_r=1$
(the definition $d_r$ is in the proof of Lemma \ref{S-mono-poset}).
Thus, $L_i\in Z_{\mathcal{E}}$ for each $i\in[t]$.
Let $T_1,T_2$ and $T_3$ be routes of $L_0$ with $|T_1|\leq |T_2|\leq |T_3|$.
Then $T_1$ and $T_2$ are $2$-paths and $f(\mathcal{E},0,r)$ is a subpath of $T_3$ with $|f(\mathcal{E},0,r)|=|T_3|-1$.
Since $L_0\neq K_{2,3}$, we have $|f(\mathcal{E},0,r)|=|T_3|-1\geq4$.
For each $L_i$, if $end(L_i)\cap I(T_j)\neq \emptyset$ for $j\in[2]$,
then $|f(\mathcal{E},0,i)|=2<|f(\mathcal{E},l,r)|$, a contradiction;
if $end(L_i)=end(T_3)$, then $f(\mathcal{E},0,i)$ is an even path, a contradiction.
Thus, $f(\mathcal{E},0,i)$ is a proper subpath of $T_3$
and $|f(\mathcal{E},0,i)|=|f(\mathcal{E},0,r)|$ for each $i\in[t]$.
If $end(L_i)\neq end(L_r)$ for $i,j\in[t]$,
then $end(L_i)\cap I(f(\mathcal{E},0,r))\neq \emptyset$
and $f(\mathcal{E},0,i)$ is not a proper subpath of $f(\mathcal{E},0,r),$
i.e., $\mathcal{E}$ does not have property {\bf R}, a contradiction.
Therefore, $end(L_i)=end(L_j)$ for each $i,j\in[t]$.
Let $H=T_2\cup T_3\cup (\bigcup_{i\in[t]}L_i)$.
Then $H$ is a graph constructed in Remark 1.
Thus, $md(H)=\frac{|H|}{2}$.
Suppose $\Gamma$ is an extremal $MD$-coloring of $H$ (see Remark 1).
Let $T_1=ue_1ae_2v$ and $T_2=uf_1bf_2v$.
Since $G=H\cup T_1$, let $\Gamma'$ be an edge-coloring of $G$ such that
$\Gamma(e)=\Gamma'(e)$ for each $e\in E(H)$, and $\Gamma(e_1)=\Gamma'(f_2)$ and $\Gamma(e_2)=\Gamma'(f_1)$.
Then $\Gamma'$ is an $MD$-coloring of $G$ with $\left\lfloor\frac{n}{2}\right\rfloor$
colors, i.e., $md(G)=\left\lfloor\frac{n}{2}\right\rfloor$.

If $L_r$ is an edge, then replace $L_l$ by
$L_l\cup L_r-I(f(\mathcal{E},l,r))$
and replace $L_r$ by $f(\mathcal{E},l,r)$.
Then the new ear-decomposition also satisfies all the conditions.
Moreover, $d_r$ is maximum and $|f(\mathcal{E},l_r,r)|=2$ is minimum in the new ear-decomposition.
Since $L_r$ is not an edge in the new ear-decomposition, this case has been discussed above.
\end{proof}

\begin{remark} Recalling the proof of Lemma \ref{S-mono-normal}, we can find a normal ear-decomposition for a given $2$-connected graph in polynomial time. For a normal ear-decomposition $\mathcal{E}$ of $G$, deciding whether $\mathcal{E}$ satisfies all the conditions of Theorem \ref{S-mono-mdex} can be done in polynomial time.
Thus, given a $2$-connected graph $G$, deciding whether $md(G)=\left\lfloor\frac{n}{2}\right\rfloor$
is polynomially solvable.
\end{remark}

\begin{corollary}
If $G$ is a $2$-connected graph with
$md(G)=\left\lfloor\frac{|G|}{2}\right\rfloor$, then $G$ is a planar graph.
\end{corollary}
\begin{proof}
By Theorem \ref{S-mono-mdex}, there is a standard ear-decomposition
$\mathcal{E}=\{L_0;L_1,\cdots,L_t\}$ of $G$
that has properties {\bf Q} and {\bf R}.
Since $G$ is a planar graph if $G$ is a cycle,
an umbrella or a $\theta$-graph, the result holds for $t=0$. Our
proof proceeds by induction on $t$. Suppose $t>0$. By Lemma \ref{S-mono-poset}, there are integers $k,i$ such that $f(\mathcal{E},k,i)$ is a path of order at least two, and $d_G(u)=2$
for each $u\in I(f(\mathcal{E},k,i))\cup I(L_i)$.
Let $G'$ be a graph obtained from $G$ by removing $L_i$.
By Lemma \ref{2-con-eard} (1), $md(G')=\left\lfloor\frac{|G'|}{2}\right\rfloor$.
By the induction hypothesis, $G'$ is a planar graph.
Since $d_G(u)=2$ for each $u\in I(f(\mathcal{E},k,i))$, there is a face $F$ of $G'$ such that $f(\mathcal{E},k,i)$ is a subpath of $F$.
Therefore, $L_i$ can be embedded in $F$ and $G$ is a planar graph.
\end{proof}


\begin{thebibliography}{1}

\bibitem {BCJLW} X. Bai, Y. Chen, M. Ji, X. Li, Y. Weng, W. Wu,
Proper disconnection of graphs, arXiv:1906.01832 [math.CO].

\bibitem {BL} X. Bai, X. Li,
Graph colorings under global structural conditions, arXiv:2008.07163 [math.CO].

\bibitem {BBY}
J. Bang-Jensen, T. Ballitto, A. Yeo,
Proper-walk connection number of graphs, {\it J. Graph Theory}, DOI: 10.1002/jgt.22609,
in press (2020), 1--23.

\bibitem {B}
J.A. Bondy, U.S.R. Murty, Graph Theory, GTM 244, Springer,
2008.

\bibitem {BFGM} V. Borozan, S. Fujita, A. Gerek, C. Magnant, Y. Manoussakis, L. Montero, Zs. Tuza,
Proper connection of graphs, {\it Discrete Math}. {\bf 312(17)} (2012), 2550--2560.

\bibitem{CY} Y. Caro, R. Yuster,
Colorful monochromatic connectivity, {\it Discrete Math.} {\bf 311} (2011), 1786--1792.

\bibitem {CDHHZ} G. Chartrand, S. Devereaux, T.W. Haynes, S.T. Hedetniemi, P. Zhang,
Rainbow disconnection in graphs, {\it Discuss. Math. Graph Theory} {\bf 38(4)} (2018), 1007--1021.

\bibitem{CJMZ} G. Chartrand, G.L. Johns, K.A. McKeon, P. Zhang,
Rainbow connection in graphs, {\it Math. Bohem.} {\bf 133} (2008), 85--98.

\bibitem{CLLW}
Y. Chen, P. Li, X. Li, Y. Weng,
Complexity results for the proper disconnection of graphs,
{\it Proceedings of 14th International Frontiers
of Algorithmics Workshop (FAW 2020), LNCS} No.12340.

\bibitem {PSC} P. Dankelmann, S. Mukwembi, H.C. Swart,
Average distance and vertex-connectivity, {\it J. Graph Theory} {\bf 62(2)} (2010), 157--177.

\bibitem {FKM} O. Favaron, M. Kouider, M. Mah\'{e}o,
Edge-vulnerability and mean distance, {\it Networks} {\bf 19} (1989), 493--504.

\bibitem {LL} P. Li, X. Li,
Monochromatic disconnection of graphs, accepted for publication in {\it Discrete Appl. Math.} arXiv:1901.01372 [math.CO].

\bibitem {LLi} P. Li, X. Li,
Monochromatic disconnection: Erd\H{o}s-Gallai-type problems and product graphs, arXiv:1904.08583 [math.CO].

\bibitem {P} J. Plesn\'{l}k,
On the sum of all distances in a graph or digraph, {\it J. Graph Theory} {\bf 8} (1984), 1--24.

\end{thebibliography}
\end{document}